\documentclass{article}

\usepackage{graphicx}
\usepackage{amsmath}
\usepackage{amsfonts}
\usepackage{amsthm}

\usepackage{algorithm,algorithmic}

\DeclareMathOperator{\Tr}{Tr}

\newtheorem{theorem}{Theorem}
\newtheorem{remark}{Remark}

\newcommand{\pder}[2][]{\frac{\partial#1}{\partial#2}}

\title{Combination Properties \\ of Weakly Contracting Systems}
\author{Ian R. Manchester and Jean-Jacques E. Slotine}
\date{}

\begin{document}
\maketitle

\begin{abstract}

Evaluating the dimension of attractors for autonomous nonlinear
dynamical systems has a distinguished history. In this paper we recast
these classical results in the context of contraction theory. In
particular, we explore the role of Riemannian metrics, quantify
convergence rates and the effects of averaging, and derive general
system combination properties. Many of these results extend
straightforwardly to non-autonomous systems, and suggest potential
applications in computational biology and systems neuroscience.

\end{abstract}

\section{Basic results}

Consider an autonomous system with state $x\in \mathbb R^n$ and dynamics
\begin{equation}\label{eq:sys}
\dot x = f(x)
\end{equation}
We are interested in the limiting behaviour of such a system as $t
\rightarrow + \infty$, and how this behaviour is preserved under
interconnection. We assume that the system evolves on a compact and
simply-connected strictly forward-invariant set $\mathcal X$. That is,
any solution starting on the boundary of $\mathcal X$ at $t=0$ remains
in the interior of $\mathcal X$ for $t\ge0$. The matrix $J(x) =
\pder[f]{x}$ denotes the system's Jacobian, and $J_s(x) =
\frac{1}{2}(J(x)+J(x)')$ its symmetric part.

For a symmetric $n\times n$ matrix $H$, we define $\lambda_1(H) \ge \lambda_2(H) \ge  ... \ge \lambda_n(H)$ as the eigenvalues of $H$ ranked in non-increasing order. We also define a function giving the sum of the $k$ largest eigenvalues:
\[
S_k(H)\ =\ \sum_{j=1}^k \lambda_j(H)
\]
We are interested in systems satisfying the following property:
\[
\forall x \in \mathbb R^n \ \ \ \ \ \ \ \ \ S_k(J_s(x))<0
\]
Such a property allows one to bound the Hausdorff dimension of an attractor of \eqref{eq:sys}, and to our knowledge was first investigated by Douady and Oesterl{\' e} \cite{douady1980dimension}, and subsequently studied by Smith \cite{smith1986some} and Leonov and colleagues (see, e.g., \cite{leonov1996frequency}, \cite{boichenko2005dimension}). Extensions include non-autonomous systems and the addition of a storage function to reduce conservatism.

In particular, if $k=2$, i.e. $\lambda_1(J_s)+\lambda_2(J_s)<0$, then
all bounded trajectories converge to an equilibrium (not necessarily
unique), i.e. the dimension of any attractor is zero. Following
Leonov, we refer to this property of eigenvalues as {\em weak
  contraction} and to a system having this property as a {\em weakly
  contracting} system.

By contrast, a contracting system (in the identity metric) has the
property that $S_1(J_s(x))<0$, i.e. the largest eigenvalue $\lambda_1(J_s(x))$ is negative \cite{lohmiller1998contraction}.

Convergence to an equilibrium is unchanged by coordinate
transformation, so we will also call a system weakly contracting if
there exists a constant nonsingular matrix $ \Theta \ $ such that
$S_2(\Theta J(x) \Theta^{-1} +( \Theta J(x) \Theta ^{-1})')\ <0\ $ for
all $x$.  When we need to be specific, we say that the system is {\em
  weakly contracting under $T$}\cite{boichenko2005dimension}.  More
generally, as in \cite{lohmiller1998contraction}, one can consider
transformations $\Theta(x,t)$, where $\Theta(x,t)^T \Theta(x,t)$ is a
uniformly positive definite Riemannian metric, and the system is
weakly contracting if $S_2(F_s)< 0$ for all $x$, where $ \ F \ $ is the
{\em generalized Jacobian}~\cite{lohmiller1998contraction} associated
with $\Theta(x,t)$,
\[
F\ = \ \Theta J \Theta^{-1} \ + \ \dot{\Theta} \Theta^{-1}
\]
In the yet more general case where the metric transformation $
\ \Theta(x,t) \ $ is complex, the condition becomes that $ \ S_2(F_H)<
0 \ $ for all $x$, where $ \ F_H \ $ is the Hermitian part of the
generalized Jacobian $ F $ and the Riemannian metric is $\Theta(x,t)^{*T} \Theta(x,t)$.

Note that if it is known that the system has a unique equilibrium and
trajectories are bounded, then the condition that
$\lambda_1+\lambda_2<0$ in some metric (i.e., $ \ S_2(F_H)< 0 \ $ for
some $\Theta(x,t)$ ) guarantees global convergence to that equilibrium,
using a weaker condition than the full contraction condition
$\lambda_1 < 0$ (i.e., $ \ S_1(F_H)< 0 \ $ ) in that metric.

\begin{remark} For computational methods, it is important to note that $S_k(H)$ has a representation, due to Ky Fan, as
\[
S_k(H) = \max_{V\in \mathcal V_k}\Tr(VHV')
\]
where $\mathcal V_k$ is the set of all $k\times n$ matrices with orthonormal rows, i.e. satisfying $VV'$ is the identity matrix of size $k$. From this construction, it is clear that $S_k$ is convex, since it is the maximum of an infinite family of functions linear in $M$. In fact, $S_k(H)$ can be represented as a linear matrix inequality \cite[p. 238]{nesterov1994interior}.

Intuitively, in the above equation $ \ VHV' \ $ is the projection of the symmetric
matrix $H$ on an orthonormal basis spanning a $k$-dimensional subspace of
$R^n$, and the $\max$ operation picks the basis aligned with the
eigenvectors corresponding to the largest eigenvalues of $H$.

\end{remark}

\subsection{Example}
A simple damped pendulum
\[
\ddot x + b\dot x + \sin x = 0
\]
naturally has a cylindrical phase space with two equilibria, however if the angle $x$ is considered as a real number then there are equilibria at $k\pi$, $k \in \mathbb Z$. Defining a state $[x, \dot x]'$ yields the Jacobian
\[
J = \begin{bmatrix}0& 1 \\ -\cos x & -b \end{bmatrix}
\]
Since the system is two-dimensional, the condition on the eigenvalues
of $J_s$ is directly given by the trace of $J_s$, which is simply the trace of $J$,
\[
\lambda_1+\lambda_2 \ = \ -b\ < \ 0
\]
Hence every system trajectory converges to an equilibrium. Note that any system
trajectory starting at one of the unstable equilibria remains there,
while any other trajectory converges to a stable equilibrium.

Note that one can easily shows first that all system trajectories are bounded,
as the theorem requires.

For planar systems, weak contraction is related to the Bendixson
criterion for non-existence of limit cycles, but of course weak
constraction extends to higher-order systems.

\section{Combination Properties}\label{combinations}

In the spirit of \cite{lohmiller1998contraction}, we remark that it is straightforward to show that certain combinations of contracting and weakly contracting systems preserve the weak contraction property, and therefore the have the property that all solutions converge to an equilibrium. We assume that the combination does not result in the system leaving the set on which the weak contraction condition holds.

\subsection{Parallel Interconnection}

\begin{theorem}
Assume that two systems $f_a$ and $f_b$ are weakly contracting under the same metric transformation $ \Theta $, then so is the system
\[
\dot x = \alpha f_a(x) + \beta f_b(x)
\]
where $\alpha, \beta$ are non-negative constants, and $\alpha+\beta>0$.
\end{theorem}
\begin{proof}
The symmetric part of the Jacobian is
\[
J_s^{ab} = \alpha J_s^a+\beta J_s^b
\]
so that  $J_s^{ab}<0 \Leftrightarrow J_s^{ab}/(\alpha+\beta)\ <\ 0$. But the latter is a convex combination of $J_s^a$ and $J_s^b$ both of which are negative by assumption, and $S_2$ is a convex function, $J_s^{ab}<0$.
\end{proof}

By recursion, this can be extended to non-negative combinations of any number of weakly contracting systems.

\subsection{Skew-Symmetric Feedback Interconnection}

Consider a skew-symmetric interconnection of a
weakly contracting system $f_a$ and a contracting system $f_b$, with
the Jacobian
\[
J^{ab}(x) = \begin{bmatrix}J^a(x) & G(x) \\ -G(x)' & J^b(x) \end{bmatrix}
\]
for arbitrary $G(x)$. We define $\lambda_i^a = \lambda_i(J_s^a)$ and likewise for $\lambda_i^b$.

\begin{theorem}\label{thm:fbk} If $\lambda_1^a+\lambda_1^b<0$ then the feedback interconnection is weakly contracting.
\end{theorem}
\begin{proof}

The symmetric part of the interconnection's Jacobian is
\[
J^{ab}_s(x) = \begin{bmatrix}J^a_s(x)& 0\\ 0 &J^b_s(x)\end{bmatrix}
\]
and so each eigenvalue $\lambda_i^a$, $\lambda_j^b$, $i = 1, 2, ...$, $j= 1, 2, ... $ of $J^a_s(x), J^b_s(x)$ is an eigenvalue of $J^{ab}_s$

Either $\lambda_1^b \le \lambda_2^a$ or $\lambda_1^b > \lambda_2^a$. In the first case the combined system has the same $\lambda_1+\lambda_2$ as system $a$. In the second case, $\lambda_1^b$ becomes the new $\lambda_2^{ab}$ and so $\lambda_1^a+\lambda_1^b<0$ implies weak contraction.
\end{proof}

Note from the proof that a less restrictive (although possibly
less convenient) test would be that the sum of the two largest of
$\lambda_1^a, \lambda_2^a, \lambda_1^b, \lambda_2^b$ should be
strictly negative.

Following~\cite{tabareau2006notes}, the result can easily be extended
to more complex feedback interconnections by using metric
transformations. In particular, it extends immediately to any constant
loop gain $ \ k > 0$ ,
\[
J^{ab}(x) = \begin{bmatrix}J^a(x) & k \ G(x) \\ -G(x)' & J^b(x) \end{bmatrix}
\]
by using the metric transformation
\[
\Theta = \begin{bmatrix}I_n & 0\\0 & \sqrt k \ I_m\end{bmatrix}
\]
where $n$ and $m$ are the state dimensions of systems $a$ and $b$.

The results extend by recursion to similar feedback combinations of
any number of systems.

\subsection{Hierarchical Interconnections}

Consider a hierarchical connection of systems with the Jacobian
\[
J^{ab}(x) = \begin{bmatrix}J^a(x) & G(x) \\ 0 & J^b(x) \end{bmatrix}
\]
for arbitrary $G(x)$, with either system $a$ or $b$ weakly contracting, and the other contracting.

\begin{theorem} If $\lambda_1^a+\lambda_1^b<0$ then the hierarchical interconnection is weakly contracting.
\end{theorem}
\begin{proof}
Let $n$ and $m$ be the state dimensions of systems $a$ and $b$. Consider the family of transformations
\[
\Theta = \begin{bmatrix}I_n & 0\\0 & \epsilon I_m\end{bmatrix}
\]
where $\epsilon>0$ is a parameter, and $I_r$ is the $r$-dimensional identity matrix. The symmetric part of the interconnection's Jacobian is
\[
J^{ab}_s(x) = \begin{bmatrix}J^a_s(x)&\frac{\epsilon}{2} G(x)\\ \frac{\epsilon}{2} G(x)' &J^b_s(x)\end{bmatrix}
\]
and clearly for each $x$ there is a sufficiently small $\epsilon$ such that the eigenvalues of $J^{ab}_s$ can be brought arbitrarily close to the eigenvalues of the matrix
\[
\begin{bmatrix}J^a_s(x)& 0\\ 0 &J^b_s(x)\end{bmatrix}
\]
and since $\mathcal X$ is compact, one can find an $\epsilon$ that works uniformly for $x\in \mathcal X$. The remainder of the proof follows exactly that of Theorem \ref{thm:fbk}.
\end{proof}

The remarks following the skew-symmetric interconnection also apply here.

By recursion, any combination of the above combinations yields a
weakly contracting system, and therefore every bounded trajectory of
the overall system tends to an equilibrium.

%
%
%
%
%
%

\section{Extensions}

In this section, we consider more generally a {\em non-autonomous} system
\begin{equation}\label{eq:sys_na}
\dot x = f(x,t)
\end{equation}

\subsection{Exponential convergence of $i$-dimensional volumes}

Proceeding exactly as in~\cite{leonov1996frequency,boichenko2005dimension}, one can show
\[
\frac{d}{dt}\|\delta z_1 \wedge  \delta z_2\| \ \le \ (\lambda_1+\lambda_2)\ \|\delta z_1 \wedge \delta z_2\|
\]
where each $\ \delta z \ $ now represents a differential displacement
weighted by the metric transformation $\Theta(x,t)$, that is $\ \delta
z = \Theta(x,t) \delta x \ $ \cite{lohmiller1998contraction}, the $
\lambda_j \ $ are the ranked eigenvalues of the symmetric (or
Hermitian) part of the generalized Jacobian associated with
$\Theta(x,t)$, and $ \ \wedge \ $ denotes the vector product.  The
above implies that if $ \ \lambda_1+\lambda_2\ $ is upper bounded by
some negative constant $ -\alpha$, then the {\em area} of differential
surfaces shrinks exponentially with rate $\alpha$. If this condition
on the largest two eigenvalues is verified in some metric, we will
also call the non-autonomous system {\em weakly contracting}.

As in~\cite{leonov1996frequency,boichenko2005dimension}, the result
extends to sub-volumes of higher dimensions, all the way to the familiar Gauss theorem
relating rate of change of volume to system divergence,
\[ \forall i = 1,..., n \ \ \ \ \ \ \ \ \ \ \ \ 
\frac{d}{dt}\|\delta z_1 \wedge ... \wedge  \delta z_i\| \ \le \ (\lambda_1+ .... + \lambda_i)\ \|\delta z_1 \wedge ... \wedge  \delta  z_i\|
\]
where $ \ \wedge \ $ denotes more generally the outer product and 
$ \|\delta z_1 \wedge ... \wedge  \delta z_i\| \ $  measures 
the volume of the polytope constructed on the $ \ \delta z_j \ $ \cite{Schwartz}.

Note that in the case of weakly contracting {\em autonomous} systems,
the result in~\cite{leonov1996frequency} (for identity or constant metric) 
is more precise $ - $ not only all surfaces shrink, but actually any bounded trajectory
asymptotically tends to some equilibrium (and not just to some $1$-dimensional
manifold).

\subsection{Storage function}

As noticed earlier, \cite{leonov1996frequency,boichenko2005dimension} also suggest
the introduction of a storage function to reduce conservatism and obtain more general
bounds. This is also the approach of \cite{Pogromsky} in the more specific application
of convergence to trajectories. In the context of this paper, the same modification 
can be derived as follows. 

First, consider a metric transformation of the simple form $ \Theta_e
(t) = e^{\gamma(t)} \ I\ $, where $ \gamma(t) $ is a differentiable
scalar function. This leads to the generalized Jacobian
\[\Theta_e J \Theta_e^{-1} \ + \ \dot{\Theta}_e \Theta_e^{-1} \ = \ J(x) \ + \ \dot \gamma(t)\  I
\]
More generally, given an original $\ \Theta(x,t) \ $, an augmented  metric transformation 
of the form 
\[ \Theta_e (x,t) = e^{\gamma(x,t)} \ \Theta(x,t) 
\]
yields an extra term $\ \dot \gamma\ I \ $ in the generalized Jacobian $ \ F_e \ $,
\[F_e \ = \ \Theta_e J \Theta_e^{-1} \ + \ \dot{\Theta}_e \Theta_e^{-1} \ = \ \Theta J \Theta^{-1} \ + \ \dot{\Theta} \Theta^{-1} \ + \  \dot \gamma\  I\ = \ F + \  \dot \gamma\  I
\]

Assume now that the function $\gamma(x,t)$ we chose is {\em bounded}, 
similarly to \cite{leonov1996frequency,boichenko2005dimension,Pogromsky}. 
By construction, letting
\[ \delta z = \Theta_e \delta x
\]
we have
\[\frac{d}{dt} \ \delta z \ \le \ \lambda_1(F_e) \ \delta z
\]
so that
\[\forall t \ge 0\ \ \ \ \ \ \ \ \ \|\delta z(t)\| \ \le \  \|\delta z(t=0)\| \ e^{\ \int_o^t  \lambda_1(F_e)(t)\ dt}
\]
Since $\gamma(x,t)$ is bounded, in the exponential above it will
be dominated by the original contraction rate,
\[
\int_o^t  \lambda_1(F_e)(t)\ dt \ = \ \int_o^t \lambda_1(F)(t)\ dt\ +\ \gamma(x,t)\ -\ \gamma(x(0),0)
\]
and thus as $\ t \rightarrow +\infty \ $
the contraction rate computed from $ \ F_e \ $ will be the same as the contraction rate
computed from $ \ F \ $.

These results extend straightforwardly to the weaker forms of
contraction discussed above. 

\begin{theorem} Consider a non-autonomous system $ \ \dot x = f(x,t) \ $.
Assume there exist a metric transformation $\Theta(x,t)$, and
a differentiable bounded scalar function $ \ \gamma(t)$, such that
\[ \exists \ \alpha > 0 \ \ \ \ \ \ \forall t \ge 0\ \ \ \ \ \ \ \ \ \ ( \lambda_1 \ + \ .... \ + \ \lambda_i) \ + \ {\dot \gamma} \ \le \ -\ \alpha
\]
where the $ \lambda_j $ are the ranked eigenvalues of the Hermitian
part of the generalized Jacobian associated with $\Theta(x,t)$. Then
any $i$-dimensional volume shrinks exponentially to zero with rate 
$ \ \alpha $.
\end{theorem}


The above may also be applied to virtual systems in the sense of~\cite{wang2005partial},
which by nature are non-autonomous.

Also note that the system combination properties derived earlier extend straightforwardly
to non-autonomous systems.

\section{Discussion}

Just as contraction can be understood in terms of differential line
elements and Riemannian distances, weak contraction can be understood
in terms of shrinking of differential planar areas. A key insight
of~\cite{leonov1996frequency} is that, in the case of autonomous
systems, weak contraction does not just imply that all surfaces
shrink, but actually more precisely that any bounded system trajectory
tends to some {\em equilibirum point}.  This is achieved by proving
that neither limit cycles or more complex behaviors can occur in that
case.

To prove the non-existence of limit cycles~\cite{leonov1996frequency}, 
assume a that closed curve ${\mathcal C}$ is preserved
under the flow of the system, and consider the two-dimensional
submanifold of smallest area having $\mathcal C$ as a boundary (like a
film of soapy water on a loop for blowing bubbles). Since this area
must shrink under the flow of the system, this contradicts the
preservation of $\mathcal C$.  The extension to nonlinear metrics and
Riemannian area integrals discussed in this paper is natural.

The non-existence of more complex (chaotic) behaviour follows from the
fact that the weak contraction conditions are open -- i.e. the strict
inequality means they are preserved under small perturbations of $f$
-- and the Pugh closing lemma which states, roughly speaking, that any
chaotic attractor can be perturbed by a small amount to result in a
closed cycle, see~\cite{leonov1996frequency}.
%

It is interesting to note that transverse contracting systems
\cite{manchester2014transverse} satisfy $ \ \lambda_1+\lambda_2 \ =
\ \lambda_2 \ <\ 0$ for some choice of metric. Transverse contracting
systems have the property that all solutions converge to the same
limit set, which is either a unique equilibrium or a unique limit
cycle.  This does not contradict the above results, as convergence to
a limit cycle can only occur if the domain of transverse contraction
is not a simply connected set. In fact, if a system is transverse
contracting on a simply connected set, then all solutions converge to
a unique equilibrium (note again that transverse contraction is a
weaker condition than contraction on the same set, but convergence may
only be asymptotic).

Since combination properties are based on algebraic rather than
topological considerations, the fact that transverse contracting
systems satisfy $ \ \lambda_1+\lambda_2 \ < \ 0$ in some metric also
explains the similarity between the combination properties of weakly
contracting systems (section~\ref{combinations}) and transverse
contracting systems (\cite{manchester2014transverse}, section 4) when
connected to a contracting system.

In biology and robotics (e.g. locomotion), natural behaviours of
dynamic systems include convergence to an equilibrium (unique or not)
and oscillation, so weak contraction and transverse contraction can
potentially provide useful frameworks for studying interconnection of
biological systems.


Synchronisation is a special case of weak contraction in which
$\lambda_1=0$ and the rate of convergence to a synchronised state is
given by the first non-zero eigenvalue (which, for identical
subsystems of dimension $m$, is $ \lambda_{m+1}$). Contraction behavior
of the synchronized system can in turn be analyzed from the reduced
(quotient) system, where all synchronized states are collapsed to a
single state~\cite{wang2005partial}.

The use of norms other than the Euclidean norm in the above results,
similarly to the contraction context of
\cite{lohmiller1998contraction}, is a direction of future research.

\bibliographystyle{unsrt}
\bibliography{ref}

\end{document}